\documentclass[11pt]{amsart}
\usepackage{amsmath, amssymb, amsthm, amsfonts,enumitem}
\usepackage[colorlinks = true, citecolor = blue]{hyperref}

\theoremstyle{plain}
\newtheorem{theorem}{Theorem}[section]
\newtheorem{corollary}[theorem]{Corollary}
\newtheorem{lemma}[theorem]{Lemma}
\newtheorem{proposition}[theorem]{Proposition}
\newtheorem{fact}[theorem]{Fact}
\newtheorem*{theorem*}{Theorem}
\newcounter{MainCorollaryCounter}
\newtheorem{MainCorollary}[MainCorollaryCounter]{Corollary}

\newcounter{MainTheoremCounter}
\newtheorem{MainTheorem}[MainTheoremCounter]{Theorem}

\theoremstyle{definition}
\newtheorem{definition}[theorem]{Definition}

\newtheorem*{setupSRF}{Setup for Section~\ref{sec.simpleRankFour}}
\newtheorem*{setupCA}{Setup for Section~\ref{sec.corollaryA}}

\newtheorem*{conjecture*}{Conjecture}
\newtheorem{remark}[theorem]{Remark}

\newcommand{\modu}[1]{\overline{#1}}
\newcommand{\textdef}[1]{\textit{#1}}
\newcommand{\orbit}{\mathcal{O}}

\DeclareMathOperator{\rk}{rk} 
\DeclareMathOperator{\m}{m}

\DeclareMathOperator{\gl}{GL} 
\DeclareMathOperator{\aff}{AGL}
 
\DeclareMathOperator{\psl}{PSL} 
\DeclareMathOperator{\pssl}{(P)SL} 
 
\DeclareMathOperator{\proj}{\mathbb{P}} 

\DeclareMathOperator{\emorph}{End} 
 
\DeclareMathOperator{\points}{\mathcal{P}} 
\DeclareMathOperator{\lines}{\mathcal{L}}

\begin{document}
\title[Groups of Morley rank $4$]{Groups of Morley rank $4$}
\author{Joshua Wiscons}
\address{Department of Mathematics\\
Hamilton College\\
Clinton, NY 13323, USA}
\email{jwiscons@hamilton.edu}
\thanks{This material is based upon work supported by the National Science Foundation under grant No. OISE-1064446.}
\keywords{finite Morley rank, simple group, generic transitivity}
\subjclass[2010]{Primary 03C60; Secondary 20B22}
\begin{abstract}
We show that any simple group of Morley rank $4$ must be a bad group with no proper definable subgroups of rank larger than $1$. We also give an application to groups acting on sets of Morley rank $2$.
\end{abstract}
\maketitle

%%% SECTION %%%%%%%%%%%%%%%%%%%%%%%%%%%%%%%%%%%%%%%%%%%%%%%%%
\section{Introduction}
%%%%%%%%%%%%%%%%%%%%%%%%%%%%%%%%%%%%%%%%%%%%%%%%%%%%%%%%%
%%%%%%%%%%%%%%%%%%%%%%%%%%%%%%%%%%%%%%%%%%%%%%%%%%%%%%%%%
This note investigates the structure of groups of Morley rank $4$; in what follows, rank always refers to Morley rank. Our motivation is two-fold. On one hand, a ``classification'' of groups of rank $4$ has direct applications to groups acting on sets of small rank. This was our initial point of view, and it is responsible for our inclusion of Corollary~\ref{cor.B}. More precisely, we are interested in applying knowledge of groups of rank $4$ to the exploration of generically sharply $n$-transitive actions on sets of rank $2$. An effort to understand these actions was initiated by Gropp in \cite{GrU92}, and their consideration sits inside of a larger project, started by Borovik and Cherlin  \cite{BoCh08}, to find a natural bound on the degree of generic transitivity for definably primitive permutation groups of finite Morley rank that depends only on the rank of the set being acted upon.

Our other reason for studying groups of rank $4$ is to add ever-so-slightly to the evidence that the Algebraicity Conjecture may hold for groups with involutions. The Algebraicity Conjecture posits that every infinite simple group of finite Morley rank is isomorphic to an algebraic group over an algebraically closed field, but the unresolved possibility of a so-called ``bad group'' of rank $3$ leaves the conjecture on shaky ground. However, it is known that any simple group of rank at most $3$ \emph{that has an involution} is indeed algebraic, and our main result extends this to rank $4$. But, with regards to the Algebraicity Conjecture, it is not so much our main result, but rather the method of proof, that is important. The method, which exploits the geometry of involutions in a potential counterexample, has previously illuminated various other difficult configurations, e.g. the structure of bad groups and the structure of sharply $2$-transitive groups, and as such, our work on this problem seems to suggest that a study of this geometry may be useful in advancing the general theory of groups with low Pr\"ufer $2$-rank. 

\begin{MainTheorem}\label{thm.A}
Any simple group of Morley rank $4$  is a bad group with no proper definable subgroups of rank larger than $1$. In particular, there are no simple groups of Morley rank $4$ with involutions.
\end{MainTheorem}

We stress that the proof of Theorem~\ref{thm.A} is relatively elementary with the main tool being Hrushovski's characterization of groups acting transitively on strongly minimal sets. Although we do employ a handful of other not so trivial results (see Section~ \ref{sec.prelim} for example), we do not use the classification of the even and mixed type simple groups nor do we use the theory of minimal simple groups, e.g. \cite[Th\'eor\`eme-Synth\`ese]{DeA07}. 

Our approach is as follows. Let $G$ be a simple group of Morley rank $4$. By Fact~\ref{fact.Hru}, a  proper definable connected subgroup $B<G$ has rank at most $2$, and we ``only'' need to show that $\rk B = 2$ gives rise to a contradiction. If such a $B$ exists,  then the action of $G$ on the right cosets of $B$ is virtually definably primitive, and we  use this to prove that, for an appropriate choice of $B$, the action is quite close to being sharply $2$-transitive, see the remarks preceding Lemma~\ref{lem.TwoTransTwoPointStab}. We will have seen that $B$ contains involutions, so in the case that the action is honestly sharply $2$-transitive, \cite[Proposition~11.71]{BoNe94} implies that $B$ has a normal complement, contradicting the simplicity of $G$. When sharp $2$-transitivity is out of reach, it is only barely out of reach, and we exploit a similar geometrical approach as in the proof of  \cite[Proposition~11.71]{BoNe94}. 

Theorem~\ref{thm.A}  yields the following corollary delineating the structure of groups of rank $4$ according to the rank of their Fitting subgroup. The \textdef{Fitting subgroup} of a group $G$ is the subgroup $F(G)$ generated by all normal nilpotent subgroups. Recall that a group is said to be \textdef{quasisimple} if it is perfect, and modulo its center, it is simple. Also, for a group $G$, we write $G=A*B$ if $A$ and $B$ are commuting subgroups that generate $G$, i.e. $G$ is the \textdef{central product} of $A$ and $B$.

\begin{MainCorollary}\label{cor.A}
Let $G$ be a connected group of Morley rank $4$.
\begin{enumerate}
\item If $\rk F(G) \ge 2$, then $G$ is solvable.
\item If $\rk F(G) = 1$, then either
\begin{enumerate}
\item $G$ is a quasisimple bad group, or
\item $G = F(G) * Q$ for some quasisimple subgroup $Q$ of rank $3$.
\end{enumerate}
\item If $\rk F(G) = 0$, then either
\begin{enumerate}
\item $G$ is a quasisimple bad group, or
\item $G$ has a normal quasisimple bad subgroup of rank $3$. 
\end{enumerate}
\end{enumerate}
In particular, $\rk F(G) \ge 1$ whenever $G$ has an involution.
\end{MainCorollary}

Finally, we give an application to groups of rank $4$ with a generically $2$-transitive action on a set of rank $2$; we show that such groups are either solvable or ``approximately'' $\gl_2$. A definable action of a group of finite Morley rank $G$ on a definable set $X$ is said to be \textdef{generically $n$-transitive} if $G$ has an orbit $\mathcal{O}$ on $X^n$ such that the rank of $X^n - \mathcal{O}$ is strictly less than the rank of $X^n$. 

\begin{MainCorollary}\label{cor.B}
If $G$ is a connected nonsolvable group of Morley rank $4$ acting faithfully, definably, transitively, and generically $2$-transitively on a definable set of rank $2$, then there is an algebraically closed field $K$ for which $G =Z(G) \cdot Q$ with  $Z(G)\cong K^\times$ and $Q\cong \pssl_2(K)$. Additionally, the action of $Q/Z(Q)$ on the $Z(G)$-orbits is equivalent to $(\proj^1(K),\psl_2(K))$.
\end{MainCorollary}

%%% SECTION %%%%%%%%%%%%%%%%%%%%%%%%%%%%%%%%%%%%%%%%%%%%%%%%%
\section{Preliminaries}\label{sec.prelim}
%%%%%%%%%%%%%%%%%%%%%%%%%%%%%%%%%%%%%%%%%%%%%%%%%%%%%%%%%
%%%%%%%%%%%%%%%%%%%%%%%%%%%%%%%%%%%%%%%%%%%%%%%%%%%%%%%%%
We collect some background results for our analysis; the general theory of groups of finite Morley rank can be found in  \cite{PoB87}, \cite{BoNe94}, and \cite{ABC08}. 

%%% SUBSECTION %%%%%%%%%%%%%%%%%%%%%%%%%%%%%%%%%%%%%%%%%%%%%%%%%
\subsection{Groups of small rank}
%%%%%%%%%%%%%%%%%%%%%%%%%%%%%%%%%%%%%%%%%%%%%%%%%%%%%%%%%

\begin{definition}\label{def.definableGset}
We call $(X,G)$ a \textdef{permutation group} if $G$ is a group acting faithfully on the set $X$, and we say that $(X,G)$ has finite Morley rank if $G$, $X$, and the action of $G$ on $X$ are all definable in some ambient structure of finite Morley rank.
\end{definition}

\begin{fact}[Hrushovski, see \protect{\cite[Theorem~11.98]{BoNe94}}]\label{fact.Hru}
Let $(X,G)$ be a transitive permutation group of finite Morley rank with $X$ of rank $1$ and (Morley) degree $1$. Then $\rk G \le 3$, and if $\rk G >1$, there is an interpretable algebraically closed field $K$ such that either
\begin{enumerate}
\item $(X,G)$ is equivalent to $(K,\aff_1(K))$, or
\item $(X,G)$ is equivalent to $(\proj^1(K),\psl_2(K))$.
\end{enumerate}
\end{fact}

Next, we gather some information about groups of rank $2$. It should be noted that our definition of a unipotent group is rather coarse and not standard, but in the case of low Morley rank, it will suffice. 

\begin{definition}
Let $G$ be a group of finite Morley rank. Then
\begin{enumerate}
\item $G$ is called a \textdef{decent torus} if $G$ is divisible, abelian, and equal to the definable hull of its torsion subgroup, and
\item $G$ is said to be \textdef{unipotent} if $G$ is connected, nilpotent, and does not contain a nontrivial decent torus. 
\end{enumerate}
\end{definition}

\begin{fact}[\cite{ChG79}]\label{fact.rankTwoGroups}
Let $B$ be a connected group of rank $2$. Then $B$ is solvable. If $B$ is nilpotent and nonabelian, then $B$ has exponent $p$ or $p^2$ for some prime $p$. If $B$ is nonnilpotent, then
\begin{enumerate}
\item $B = B' \rtimes T$ with $T$ a decent torus containing $Z(B)$, 
\item $B/Z(B) \cong K^+ \rtimes K^\times$ for some algebraically closed field $K$, and
\item every automorphism of $B$ of finite order is inner.
\end{enumerate}
\end{fact}

We now collect some easy consequences of the previous fact. 

\begin{lemma}\label{lem.RankTwoAbelian}
If $B$ is a connected group of rank $2$, then any one of the following implies that $B$ is abelian:
\begin{enumerate}
\item $B$ normalizes a nontrivial decent torus,
\item $B$ contains two distinct unipotent subgroups of rank $1$, or
\item $B$ is nilpotent and contains two distinct definable connected subgroups of rank $1$.
\end{enumerate} 
\end{lemma}
\begin{proof}
The first item follows immediately from the previous fact. For the third point, assume that $B$ is nilpotent and contains two distinct definable connected subgroups each of rank $1$. By the ``Normalizer Condition'' for nilpotent groups of finite Morley rank, see \cite[I,~Proposition~5.3]{ABC08}, both subgroups must be normal in $B$, and hence, both have an infinite intersection with $Z(B)$, see \cite[I,~Lemma~5.1]{ABC08}. Thus $Z(B) = B$. This establishes the third point, and the second now follows since the hypothesis implies, upon invoking Fact~\ref{fact.rankTwoGroups}, that $B$ is nilpotent. 
\end{proof}

We will say a little more about nonnilpotent groups of rank $2$ for which we need a lemma (and its corollary). This is certainly well-known.

\begin{lemma}\label{lem.DivisibleFiniteQuotient}
If $D$ is a divisible abelian group and $A$ is a finite subgroup, then $D\cong D/A$. 
\end{lemma}
\begin{proof}
The critical case is when $D \cong \mathbb{Z}_{p^\infty}$ for some prime $p$. In this case, if $A$ has order $m$, then $A$ is the unique subgroup of $D$ of order $m$, so $A$ is the kernel of the map $D\rightarrow D:x\mapsto x^m$. By divisibility, the map is surjective, so in this case, $D\cong D/A$.

Now, the general case easily reduces to the case of $A$ cyclic of prime power order, so assume that $A = \langle a \rangle$ with $a$ a $p$-element for some prime $p$. By the divisibility of $D$, $a$ is contained in a subgroup $T$ that is isomorphic to $\mathbb{Z}_{p^\infty}$. As $T$ is divisible and $D$ is abelian, it is well-known that $T$ has a complement $H$ in $D$. Now we have that $D/A \cong (T\times H)/ (A \times 1) \cong T/A \times H$, and as we have already observed that $T/A \cong T$, we are done. 
\end{proof}

\begin{corollary}\label{cor.DivisibleFiniteQuotient}
If $G$ is a connected group of finite Morley rank and $N$ is a finite normal subgroup for which $G/N \cong K^\times$ with $K$ a field, then $G\cong K^\times$.
\end{corollary}
\begin{proof}
This follows directly from the previous lemma since the hypotheses, together with \cite[I,~Lemma~3.8]{ABC08}, imply that  $G$ is divisible abelian.
\end{proof}

\begin{lemma}\label{lem.RankTwoNonnilCenter}
Let $B$ be a nonnilpotent connected group of rank $2$. Set $n:=|Z(B)|$. Then $Z(B)$ contains all elements of $B$ of order dividing $n$, so $Z(B)$ is the unique subgroup of $B$ of cardinality $n$.
\end{lemma}
\begin{proof}
Set $Z:=Z(B)$. By Fact~\ref{fact.rankTwoGroups}, $B$ is solvable, and $Z$ is finite. Further, $B=B'\rtimes T$ with $T$ a decent torus containing $Z$, and $B/Z \cong K^+\rtimes K^\times$ for some algebraically closed field $K$. By the previous corollary, $T\cong K^\times$, so $T$ contains a unique subgroup of order $m$ for every $m$ dividing $n$. Further, as $B' \cong K^+$, we see that $B'$ has no nontrivial elements of order dividing $n$.

Now, let $g\in B$ be of order $m$ with $m$ dividing $n$. As $T$ contains a unique subgroup of order $m$, we find that the image of $g$ in $B/B'$ lies in the image of $Z$. Thus, we may write $g=uz$ for some $u\in B'$ and some $z\in Z$. Now, $1 = g^m = u^mz^m$, so $u^m \in B' \cap Z$. Hence, $u^m = 1$, so our previous observation implies that $u=1$ and $g\in Z$.
\end{proof}

%%% SUBSECTION %%%%%%%%%%%%%%%%%%%%%%%%%%%%%%%%%%%%%%%%%%%%%%
\subsection{Tori}
%%%%%%%%%%%%%%%%%%%%%%%%%%%%%%%%%%%%%%%%%%%%%%%%%%%%%%%%%
Here we simply quote a pair of general facts about tori.

\begin{fact}[\protect{\cite[Theorem~1]{AlBu08},\cite[Corollary~2.12]{FrO09}}]\label{fact.centTori}
If $T$ is a decent torus in a connected group of finite Morley rank, then $C(T)$ is connected. 
\end{fact}

\begin{fact}[\protect{\cite[Theorem~3]{BuCh08}}]\label{fact.torality}
Let $p$ be a prime, and assume that $G$ is a group of finite Morley rank with no infinite elementary abelian $p$-group. Then every $p$-element of $G$ lies in a decent torus. 
\end{fact}

%%% SUBSECTION %%%%%%%%%%%%%%%%%%%%%%%%%%%%%%%%%%%%%%%%%%%%%%
\subsection{Strongly real elements}
%%%%%%%%%%%%%%%%%%%%%%%%%%%%%%%%%%%%%%%%%%%%%%%%%%%%%%%%%
We end this section with the Brauer-Fowler Theorem for groups of finite Morley rank.

\begin{definition}
An element of a group is said to be \textdef{strongly real} if it is the product of two involutions.
\end{definition}

Note that an element $r$ is strongly real if and only if it is inverted by some involution that is not equal to $r$.

\begin{fact}[\protect{\cite[Theorem~10.3]{BoNe94}}]\label{fact.BrauerFowler}
For every involution $i$ of a group of finite Morley rank $G$ there is a nontrivial strongly real element $r$ for which \[\rk G \le \rk C(r) + 2\cdot\rk C(i).\]
\end{fact}

%%% SECTION %%%%%%%%%%%%%%%%%%%%%%%%%%%%%%%%%%%%%%%%%%%%%%%%%
\section{Some permutation group theory}
%%%%%%%%%%%%%%%%%%%%%%%%%%%%%%%%%%%%%%%%%%%%%%%%%%%%%%%%%
%%%%%%%%%%%%%%%%%%%%%%%%%%%%%%%%%%%%%%%%%%%%%%%%%%%%%%%%%

As mentioned in the introduction, Theorem~\ref{thm.A} reduces to a study of virtually definably primitive permutation groups with connected point stabilizers of rank $2$. The focus of this section is the case of abelian point stabilizers; this is addressed by Proposition~\ref{prop.VprimRankAbelian}. We then conclude the section with some analogs of Proposition~\ref{prop.VprimRankAbelian} for the nonabelian case.

%%% SUBSECTION %%%%%%%%%%%%%%%%%%%%%%%%%%%%%%%%%%%%%%%%%%%%%%
\subsection{Primitivity}
%%%%%%%%%%%%%%%%%%%%%%%%%%%%%%%%%%%%%%%%%%%%%%%%%%%%%%%%%
We begin by recalling the essential definitions. For more information on the various notions of primitivity, we refer to \cite{BoCh08}.

\begin{definition}
Assume that a group $G$, a set $X$, and an action of $G$ on $X$ are all definable in some ambient structure. We say that the action is \textdef{definably primitive} if every definable (with respect to the ambient structure) $G$-invariant equivalence relation is either trivial or universal; where as, we call the action \textdef{virtually definably primitive} if every definable $G$-invariant equivalence relation either has finite classes or finitely many classes. 
\end{definition}

As with the usual notion of primitivity, the above two analogs can be described in terms of subgroups of $G$.

\begin{fact}[\protect{\cite[Lemma~1.13]{BoCh08}}]
Let $(X,G)$ be a transitive permutation group definable in some ambient structure, and fix $x\in X$. Then 
\begin{enumerate}
\item $(X,G)$ is definably primitive if and only if $G_x$ is a maximal definable subgroup of $G$, and 
\item $(X,G)$ is virtually definably primitive if and only if for every definable subgroup $H$ with $G\ge H\ge G_x$ either $|G:H|$ or $|H:G_x|$ is finite. 
\end{enumerate}
\end{fact}

A \textdef{quotient} of a permutation group $(X,G)$ is any permutation group of the form $(X/{\sim},G/K)$ with ${\sim}$ a $G$-invariant equivalence relation on $X$ and $K$ the kernel of the (induced) action of $G$ on $X/{\sim}$. We often also refer to $X/{\sim}$ together with the (not necessarily faithful) action of $G$ as a quotient of $(X,G)$, which is more-or-less harmless. An important observation to make is that every transitive permutation group of finite Morley rank has a virtually definably primitive quotient, which corresponds to a proper definable subgroup of maximal rank containing a point stabilizer.  The following fact says that we can often find a quotient that is in fact definably primitive. 

\begin{fact}[\protect{\cite[Lemma~1.18]{BoCh08}}]\label{fact.primQuotient}
Let $(X,G)$ be a transitive and virtually definably primitive permutation group of finite Morley rank with infinite point stabilizers. Then $(X,G)$ has a nontrivial (but not necessarily faithful) definably primitive quotient. 
\end{fact}

We now collect a handful of remarks on definably primitive groups. 

\begin{lemma}\label{lem.DefPrim}
Assume that $(X,G)$ is a definably primitive permutation group. Let $x,y\in X$ be distinct, and assume that $G_x \neq 1$. Then
\begin{enumerate}
\item $G_x \neq G_y$, and
\item $Z(G_x)\cap Z(G_y) = 1$.
\end{enumerate}
Further, if $(X,G)$ has finite Morley rank and $G^\circ_x \neq 1$, then $G^\circ_x \neq G^\circ_y$, so in this case, $G_x$ has a unique orbit of rank $0$, namely $\{x\}$.
\end{lemma}
\begin{proof}
Notice that ``$a\sim b$ if and only if $G_a = G_b$'' is a definable equivalence relation on $X$. Thus, definable primitivity and nontrivial point stabilizers imply that point stabilizers must be pairwise distinct. 

To see that $Z(G_x)\cap Z(G_y) = 1$, recall that definable primitivity implies that point stabilizers are maximal proper definable subgroups. 
Thus, if $g\in Z(G_x)\cap Z(G_y)$, then $C(g) \ge \langle G_x, G_y \rangle$, so $g\in Z(G)$. Since $g$ fixes a point and $G$ is transitive on $X$, we find that $g= 1$. Similarly, if $G$ has finite Morley rank and $G^\circ_x = G^\circ_y$,  then $N(G^\circ_x) \ge \langle G_x, G_y \rangle$, so $G^\circ_x$ is normal in $G$. This forces  $G^\circ_x=1$. 
\end{proof}

%%% SUBSECTION %%%%%%%%%%%%%%%%%%%%%%%%%%%%%%%%%%%%%%%%%%%%%%
\subsection{Generically regular subgroups}
%%%%%%%%%%%%%%%%%%%%%%%%%%%%%%%%%%%%%%%%%%%%%%%%%%%%%%%%%
This subsection is devoted solely to the following rather general (and rather useful) connectedness lemma. 

\begin{lemma}\label{lem.ConnectedPS}
Assume that $(X,G)$ is a transitive permutation group of finite Morley rank with $G$ connected. If some definable subgroup of $G$ has a regular and generic orbit on $X$, then  all $1$-point stabilizers of $G$ are connected.  
\end{lemma}
\begin{proof}
Assume that $H$ is a definable subgroup with a regular and generic orbit, and choose $x$ in the orbit. Set $A:=G_xH$. Now, $A$ is a definable set, and since $G_x\cap H = 1$, every element of $A$ has unique representation as $gh$ with $g\in G_x$ and $h\in H$. Thus, there is a definable bijection between $A$ and $G_x\times H$. Further, as the orbit of $H$ on $x$ is generic, we find that $\rk H = \rk G-\rk G_x$. Thus, $A$ is generic in $G$, and since $G$ is connected, it must be that $G_x$ is connected as well.
\end{proof}

%%% SUBSECTION %%%%%%%%%%%%%%%%%%%%%%%%%%%%%%%%%%%%%%%%%%%%%%
\subsection{Abelian point stabilizers}
%%%%%%%%%%%%%%%%%%%%%%%%%%%%%%%%%%%%%%%%%%%%%%%%%%%%%%%%%
The goal of this subsection is to show that virtually definably primitive actions with sufficiently large abelian point stabilizers are of one flavor; this is Proposition~\ref{prop.VprimRankAbelian}. The result is not surprising, but there are a handful of details to address. We begin with a slight generalization of a well know result about $2$-transitive groups.

\begin{lemma}\label{lem.primGenTwoTransAbelian}
If $(X,G)$ is an infinite definably primitive and generically $2$-transitive permutation group of finite Morley rank with abelian point stabilizers, then $(X,G) \cong (K,\aff_1(K))$ for some algebraically closed field $K$. 
\end{lemma}
\begin{proof}
Let $x\in X$. By Lemma~\ref{lem.DefPrim}, $G_x$ acts freely on $X-\{x\}$. Thus, every orbit of $G_x$ on $X-\{x\}$ has the same rank, so generic $2$-transitivity implies that $G_x$ is transitive on $X-\{x\}$. Hence, $(X,G)$ is sharply $2$-transitive, so $(X,G)$ is equivalent to $(K,\aff_1(K))$ for some algebraically closed field $K$; see  \cite[Proposition~11.61]{BoNe94} for example. 
\end{proof}

The next lemma provides a connectedness result essential for our proof of Proposition~\ref{prop.VprimRankAbelian}.

\begin{lemma}\label{lem.GenTwoTransAbelian}
Let $(X,G)$ be an infinite transitive and generically $n$-transitive permutation group of finite Morley rank with $n\ge 2$. If $(x_1,\ldots,x_n)$ is in the generic orbit of $G$ on $X^n$ and $H:=G_{x_1,\ldots,x_{n-1}}$ is abelian-by-finite, then $G$ is centerless, and $C(H^\circ) = H^\circ$. If, additionally, $G$ is connected, then $G_{x_1,\ldots,x_{k}}$ is connected for every $k<n$.
\end{lemma}
\begin{proof}
First note that generic $2$-transitivity implies that $X$ has degree $1$, see \cite[Lemma~1.8(3)]{BoCh08}. Let $\orbit$ be the generic orbit of $H$ on $X$. By \cite[Lemma~1.6]{BoCh08}, $N(\orbit)$ acts faithfully on $\orbit$. Now, $\orbit$ is connected, so $H^\circ$ acts transitively on $\orbit$. Thus, $C(H^\circ)$, which by assumption includes $H^\circ$, acts regularly on $\orbit$, and we find that $C(H^\circ) = H^\circ$. Consequently, $Z(G) \le H$, so  faithfulness implies that $Z(G) = 1$.

Now assume that $G$ is connected. Since $H^\circ$ acts regularly on $\orbit$, Lemma~\ref{lem.ConnectedPS} shows that $G_{x_1}$ is connected. Clearly we can iterate this argument, using $H^\circ$ at each stage, to see that $G_{x_1,\ldots,x_{k}}$ is connected for every $k<n$.
\end{proof}

We can now generalize Lemma~\ref{lem.primGenTwoTransAbelian} to the virtually definably primitive setting; here we  replace the generically $2$-transitive assumption with a condition on the rank of a point stabilizer.

\begin{proposition}\label{prop.VprimRankAbelian}
Let $(X,G)$ be an infinite transitive and virtually definably primitive permutation group of finite Morley rank with abelian point stabilizers. If $G$ is connected and the point stabilizers have rank at least that of $X$, then $(X,G) \cong (K,\aff_1(K))$ for some algebraically closed field $K$. 
\end{proposition}
\begin{proof}
Let $(X,G)$ satisfy the hypotheses of the proposition, and note that this implies that $X$ is connected. We first claim that the action is generically $2$-transitive. Fix $x\in X$. Since the point stabilizers are infinite, there is some $y\in X$ that is not fixed by $G_x^\circ$. In light of the fact that $X$ is connected, the claim will follow from the observation that $G_x \cap G_y = 1$.  Indeed, if $h \in G_x \cap G_y$, then $\langle G_x, G_y\rangle\le C(h)$, and as $G_x^\circ \neq G_y^\circ$, this shows that $C(h)$ has infinite index over $G_x$. By virtually definably primitivity and the connectedness of $G$, we find that $C(h) = G$, so as $h$ fixes a point, $h = 1$. 

Now by Fact~\ref{fact.primQuotient}, $(X,G)$ has a definably primitive quotient $(\overline{X},G)$ with finite classes. Let $M$ be the kernel of the action of $G$ on $\overline{X}$. Since the classes in the quotient are finite, $M^\circ$ fixes all of $X$, so $M$ is finite. Recalling that $G$ is connected, we find that $M$ is central, so the previous lemma shows us that in fact $M=1$. Further, since $(X,G)$ is generically $2$-transitive, $(\overline{X},G)$ is as well, and the connectedness result from the previous lemma implies that the point stabilizers from  $(\overline{X},G)$ and $(X,G)$ coincide, i.e. the classes in the quotient have cardinality $1$. Thus, $(X,G)$ is definably primitive and generically $2$-transitive with abelian point stabilizers. Lemma~\ref{lem.primGenTwoTransAbelian}  applies.
\end{proof} 

\begin{corollary}\label{cor.SimpleMaxAbelian}
Let $G$ be a simple  group of finite Morley rank and $A<G$ a maximal definable connected subgroup. If $A$ is abelian, then $2\cdot\rk A<\rk G$.
\end{corollary}

%%% SUBSECTION %%%%%%%%%%%%%%%%%%%%%%%%%%%%%%%%%%%%%%%%%%%%%%
\subsection{Groups acting on sets of rank $2$}
%%%%%%%%%%%%%%%%%%%%%%%%%%%%%%%%%%%%%%%%%%%%%%%%%%%%%%%%%
Here, in the context of groups acting on sets of rank $2$, we give a couple of approximations to Proposition~\ref{prop.VprimRankAbelian} for actions with nonabelian point stabilizers. The relevant result for the sequel is Corollary~\ref{cor.SimpleMaxCenterRankFour}.

\begin{lemma}\label{prop.VprimRankNilpotentTwoTrans}
Let $(X,G)$ be a transitive and virtually definably primitive permutation group of finite Morley rank whose point stabilizers are connected and nilpotent. Assume $\rk X = 2$. If $G$ is connected and the point stabilizers have rank at least $2$, then any definably primitive quotient of $(X,G)$ is $2$-transitive.
\end{lemma}
\begin{proof}
Fix $x\in X$. First, we show that $G_x$ has no rank $1$ orbits on $X$. From this, the lemma follows quickly. Towards a contradiction, assume that $\orbit$ is a rank $1$ orbit of $G_x$, and let $A$ be the kernel of the action of $G_x$ on $\orbit$. As $G_x$ is nilpotent, Fact~\ref{fact.Hru} implies that $\rk G_x - \rk A = 1$. Thus, for every $y\in \orbit$, $A$ has corank $1$ in $G_y$, so the Normalizer Condition, together with the assumption that $G_y$ is connected, shows that $A$ is normal in $G_y$. As $G_x \neq G_y$, virtual definably primitivity implies that $A$ is normal in $G$. This is a contradiction, so $G_x$ has no rank $1$ orbits on $X$. Thus, the only obstacle to $(X,G)$ being $2$-transitive is the possibility that $G_x$ fixes more that just $x$, but we can remove this obstacle by passing to a definably primitive quotient.

Let $(\modu{X},G)$ be a definably primitive (but possibly not faithful) quotient of $(X,G)$. Let $\modu{x}$ be the image of $x$ in $\modu{X}$. By Lemma~\ref{lem.DefPrim}, $G_{\modu{x}}$, has no orbit of rank $0$ other than $\{\modu{x}\}$. Now, if $G_{\modu{x}}$ has an orbit of rank $1$ on $\modu{X}$, then $G^\circ_{\modu{x}} = G_x$ would as well. As the classes in the quotient are finite, this would imply that $G_x$ acts nontrivially on a rank $1$ subset of $X$, and this in turn would imply that $G_x$ has a rank $1$ orbit on $X$. This contradicts our work above, so $G_{\modu{x}}$ must have no orbits of rank $1$. Now, the connectedness of $G$ implies that $X$ has degree $1$, so $G_{\modu{x}}$ acts transitively on $\modu{X} - \{\modu{x}\}$.
\end{proof}

\begin{lemma}\label{lem.vprimTwoTrans}
Let $(X,G)$ be a transitive and virtually definably primitive permutation group of finite Morley rank whose point stabilizers are connected and possess a nontrivial center. Assume $\rk X = 2$. If $G$ is connected and the point stabilizers have rank $2$, then any definably primitive quotient of $(X,G)$ is $2$-transitive.
\end{lemma}
\begin{proof}
By the previous lemma, we may assume that $G_x$ is not nilpotent. As before, we first show that $G_x$ has no orbits of rank $1$. Indeed, assume that $\orbit$ is a rank $1$ orbit, and let $A$ be the kernel of the action of $G_x$ on $\orbit$.  Since $Z(G_x) \neq 1$, Fact~\ref{fact.Hru} implies that $A$ is nontrivial. Fix $y\in \orbit$.

Now, if $A$ has rank $1$, then Lemma~\ref{lem.RankTwoAbelian} implies that $A^\circ$ is unipotent. Of course $A^\circ \le G_y$, so  Lemma~\ref{lem.RankTwoAbelian} also tells us that $A^\circ$ is normal in $G_y$. Thus, $N(A^\circ) \ge \langle G_x,G_y \rangle$, so $A^\circ$ is normal in $G$. This is a contradiction. 

Thus, $A$ is finite. Hence, $A$ is central in $G_x$, and Fact~\ref{fact.Hru} implies that $Z(G_x) = A$.  By Lemma~\ref{lem.RankTwoNonnilCenter}, $A$ is the unique subgroup of $G_x$ of cardinality $n := |A|$. But then, $A$ is also the unique subgroup of $G_y$ of cardinality $n$, so $A$ is the center of $G_y$ as well. We conclude that $A$ is central in $G$, which is again a contradiction. Thus, $G_x$ has no rank $1$ orbits, and the rest follows as in the proof of the preceding lemma.
\end{proof}

\begin{corollary}\label{cor.SimpleMaxCenterRankFour}
Let $G$ be a simple group of rank $4$ and $B < G$ a definable connected subgroup of rank $2$. If $B$ has a nontrivial center, then the action of $G$ on the coset space $N(B)\backslash G$ by right multiplication is $2$-transitive.
\end{corollary}
\begin{proof}
Since $G$ is simple, Fact~\ref{fact.Hru} implies that $G$ has no definable subgroups of rank $3$. Thus, $B$ is a maximal definable connected subgroup of $G$, so the previous lemma applies to $G$ acting on the coset space $B\backslash G$. It remains to observe that $N(B)$ is a maximal definable subgroup of $G$, so the action of $G$ on $N(B)\backslash G$ is definably primitive. 
\end{proof}

%%% SECTION %%%%%%%%%%%%%%%%%%%%%%%%%%%%%%%%%%%%%%%%%%%%%%%%%
\section{Simple groups of rank $4$}\label{sec.simpleRankFour}
%%%%%%%%%%%%%%%%%%%%%%%%%%%%%%%%%%%%%%%%%%%%%%%%%%%%%%%%%
%%%%%%%%%%%%%%%%%%%%%%%%%%%%%%%%%%%%%%%%%%%%%%%%%%%%%%%%%
We now take up the proof of Theorem~\ref{thm.A}. For the remainder of this section, $G$ denotes a simple group of rank $4$ for which there is some definable subgroup of rank $2$.

\begin{setupSRF}
Let  $G$ be a simple group of rank $4$, and assume that $G$ has a definable subgroup of rank $2$.
\end{setupSRF}

%%% SUBSECTION %%%%%%%%%%%%%%%%%%%%%%%%%%%%%%%%%%%%%%%%%%%%%%
\subsection{General remarks}
%%%%%%%%%%%%%%%%%%%%%%%%%%%%%%%%%%%%%%%%%%%%%%%%%%%%%%%%%
The simplicity of $G$ implies that $G$ has no definable subgroups of rank $3$, so Corollary~\ref{cor.SimpleMaxAbelian} yields the following important fact.

\begin{remark}\label{rem.SimpleMaxAbelianRankFour}
A definable abelian subgroup of $G$ has rank at most $1$.
\end{remark}

We now comment on tori in $G$. Certainly more could be said, specifically in regards to the Weyl group, but the following will suffice for our purposes.

\begin{lemma}\label{lem.maxTori}
If $T$ is a maximal decent torus of $G$, then $T$ is self-centralizing and contained in some nonnilpotent Borel subgroup. Further, 
$T\cong K^\times$ for some algebraically closed field $K$. 
\end{lemma}
\begin{proof}
The main point is that $G$ must contain some nonnilpotent subgroup of rank $2$. Indeed, if every connected rank $2$ subgroup of $G$ is nilpotent, then $G$ is a bad group, but Corollary~\ref{cor.SimpleMaxCenterRankFour} implies that $G$ contains involutions. Since simple bad groups do not have involutions, see \cite[Theorem~13.3]{BoNe94}, $G$ must contain some nonnilpotent connected subgroup $B$ of rank $2$. Now, $B$ is solvable and nonnilpotent, so $B$ contains a nontrivial decent torus $S$ of rank $1$. By Lemma~\ref{lem.RankTwoAbelian} and Remark~\ref{rem.SimpleMaxAbelianRankFour}, $S= C^\circ(S)$. Thus $S$ is a maximal decent torus, and Fact~\ref{fact.centTori} tells us that  $S= C(S)$. As maximal decent tori are conjugate in any group of finite Morley rank (see \cite[IV,~Proposition~1.15]{ABC08}), we may take $T$ to be $S$, and the first point is complete.

Now, by Fact~\ref{fact.rankTwoGroups}, we know that $B/Z(B) \cong K^+ \rtimes K^\times$  for some algebraically closed field $K$. Further, $Z(B) \le T$, so as $T/Z(B) \cong K^\times$, we see that $T \cong K^\times$ by Corollary~\ref{cor.DivisibleFiniteQuotient}. 
\end{proof}

%%% SUBSECTION %%%%%%%%%%%%%%%%%%%%%%%%%%%%%%%%%%%%%%%%%%%%%%
\subsection{Borel subgroups}
%%%%%%%%%%%%%%%%%%%%%%%%%%%%%%%%%%%%%%%%%%%%%%%%%%%%%%%%%
Recall that a \textdef{Borel subgroup} of a group of finite Morley rank is defined to be a maximal definable connected solvable subgroup. As connected groups of rank $2$ are solvable and $G$ has no subgroups of rank $3$, \emph{a Borel subgroup of $G$ is the same as a maximal proper definable connected subgroup.}

\begin{proposition}\label{prop.SimpleRankTwoSubRankFour}
Every Borel subgroup of $G$ is self-normalizing and nonnilpotent of rank $2$. Further, $G$ has even or odd type.
\end{proposition}
\begin{proof}
We first work to show that every definable connected rank $2$ subgroup of $G$ is nonnilpotent. 
Let $T$ be a maximal decent torus contained in some Borel subgroup. Since $T$ is self-centralizing, $T$ is generous in $G$, i.e. $\bigcup T^G$ is  generic in $G$. We refer to \cite[IV,~Section~1]{ABC08} for generalities on generosity. Now assume that $G$ has a definable connected nilpotent (and nonabelian) subgroup $A$ of rank $2$; we show that $A$ is also generous. First, note that $A$ is almost self-normalizing by rank considerations. Thus, it suffices to show that $A\cap A^g$ is trivial for every $g\in G - N(A)$, and by Corollary~\ref{cor.SimpleMaxCenterRankFour}, we already know that $A\cap A^g$ is finite. Let $g\in G - N(A)$, and assume that $a \in A\cap A^g$ is nontrivial.  Let $Z$ be the connected center of $A$; $Z$ is unipotent of rank $1$ by Fact~\ref{fact.rankTwoGroups}. Now, $C^\circ(a)$ contains $\langle Z,Z^g \rangle$, so $C^\circ(a)$ is abelian by Lemma~\ref{lem.RankTwoAbelian}. As $C^\circ(a)$ has rank $2$, this contradicts Remark~\ref{rem.SimpleMaxAbelianRankFour}, so we conclude that $A$ is generous.

Since $T$ and $A$ are generous, there is a nontrivial $a\in A$ such that $C(a)$ contains a decent torus. Of course, $C(a)$ also contains the unipotent subgroup $Z$. Thus, $C(a)$ has rank $2$, and since $C^\circ(a)$ cannot be abelian, it must be nonnilpotent. By Fact~\ref{fact.rankTwoGroups}, $A$ has exponent $p$ or $p^2$ for some prime $p$. Thus $Z$ is an elementary abelian $p$-group, so the tori in $C(a)$ are isomorphic to $K^\times$ for some algebraically closed field $K$ of characteristic $p$. As the tori in $C(a)$ are self-centralizing, they contain the $p$-element $a$, a contradiction. We conclude that every connected subgroup of $G$ of rank $2$ is nonnilpotent.

Next, assume that some connected rank $1$ subgroup $A$ of $G$ is a Borel. Since every nontrivial decent torus is properly contained in a Borel, $A$ is not a decent torus. We claim that $A$ is generous, and as before, it suffices to show that $A\cap A^g$ is trivial for every $g\in G - N(A)$. Let $g\in G - N(A)$, and assume that $a \in A\cap A^g$ is nontrivial. Then $C^\circ(a)$ is equal to $\langle A,A^g \rangle$. Since $A$ is not a decent torus, we find that $C^\circ(a)$ is rank $2$ and abelian. This is a contradiction, so $A$ is generous.  Thus, as before, there is a nontrivial $a\in A$ such that $C(a)$ contains a decent torus, and of course, $C(a)$ also contains $A$. This is a contradiction. Thus, every Borel subgroup of $G$ has rank $2$. 

We now address self-normalization. Let $B$ be a Borel subgroup of $G$. Let $T$ be a maximal decent torus of $B$, and recall that $T$ is self-centralizing by Lemma~\ref{lem.maxTori}. Using the conjugacy of maximal decent tori in $B$, a Frattini argument yields $N(B) = BN_{N(B)}(T)$, so we need only show that $N_{N(B)}(T) \subset B$. Suppose not. Since $T$ is a decent torus, $N(T)/C(T) =  N(T)/T$ is finite. Lifting torsion from $N(T)/T$ (see \cite[I,~Lemma~2.18]{ABC08}), we find an $n$ in $N_{N(B)}(T) - T$ of finite order, and by Fact~\ref{fact.rankTwoGroups}, $n$ acts on $B$ as an inner automorphism. Again using Fact~\ref{fact.rankTwoGroups}, it is not hard to see that $T$ is self-normalizing in $B$, so it must be that $n$ centralizes $T$. Since $T$ is self-centralizing in $G$, we have a contradiction. Thus, $B$ is self-normalizing.

At this point, it is also clear that $G$ is not of degenerate type, so it remains to show that $G$ cannot have mixed type.  Suppose that $G$ is of mixed type. Let $S$ be the definable closure of a $2$-torus, and let $U$ be a $2$-unipotent subgroup. We know that both $S$ and $U$ are properly contained in (different) Borel subgroups. As $S$ is a maximal decent torus, $S$ is conjugate to a decent torus normalizing $U$, but the latter torus is without $2$-torsion.
\end{proof}

As $G$ has no definable subgroups of rank $3$, the previous proposition implies that \emph{every definable subgroup of $G$ of rank $2$ is connected.}

\begin{corollary}\label{cor.torality}
Let $S \subset G$. If $\rk C(S) = 2$, then $S$ is toral.
\end{corollary}
\begin{proof}
If $\rk C(S) = 2$, then Proposition~\ref{prop.SimpleRankTwoSubRankFour} implies that $C(S)$ contains a maximal decent torus $T$, and as $T$ is self-centralizing, $S \subset T$.
\end{proof}

\begin{corollary}\label{cor.normalizerUni}
Every unipotent subgroup is contained in a unique Borel subgroup.
\end{corollary}
\begin{proof}
Let $U$ be unipotent. As every Borel subgroup has rank $2$, the structure of rank $2$ groups implies that $U$ is normal in every Borel subgroup containing it. As $G$ is simple, the normalizer of $U$ has rank at most $2$,  so the normalizer is the unique Borel containing $U$.
\end{proof}

We now use the Brauer-Fowler Theorem to find Borel subgroups of a particular form. 

\begin{lemma}\label{lem.BF}
If $G$ has even type, then the centralizer of some strongly real element has rank $2$, and if $G$ has odd type, then the centralizer of any involution has rank $2$. Thus, some Borel subgroup of $G$ has a nontrivial center.
\end{lemma}
\begin{proof}
By Fact~\ref{fact.BrauerFowler}, we find an involution $i$ and a strongly real element $r$ such that $\rk C(r) + 2\cdot\rk C(i) \ge 4$. Thus, $i$ or $r$ must have a centralizer of rank $2$. If $G$ has even type,  Corollary~\ref{cor.torality} shows that no involution has a rank $2$ centralizer, so in this case, $C(r)$ has rank $2$. Now assume $G$ has odd type and that some involution has a centralizer of rank $1$. Clearly $G$ has Pr\"ufer $2$-rank equal to $1$, so Fact~\ref{fact.torality} implies that all involutions are conjugate. Thus, every involution has a rank $1$ centralizer, so $C(r)$ has rank $2$. Let $j$ be an involution inverting $r$, and note that $j$ normalizes $C(r)$. Since $C(r)$ is not abelian, $j$ centralizes some infinite subgroup of $C(r)$, so $C(r)$ contains the connected centralizer of $j$. The connected centralizer of $j$ contains $j$ by Fact~\ref{fact.torality}, so $j$ centralizes $r$. This implies that $r$ is an involution, which contradicts the fact that $\rk C(r)=2$. 
\end{proof}

Thus, $G$ is guaranteed to have some Borel subgroup $B$ with a nontrivial center, and we may invoke Corollary~\ref{cor.SimpleMaxCenterRankFour} to see that $G$  has a $2$-transitive action on $B\backslash G$. From this point of view, the next lemma  is about the $2$-point stabilizers in such an action. 

\begin{lemma}\label{lem.TwoTransTwoPointStab}
Let $B$ be a Borel subgroup of $G$ with a nontrivial center. If $g\notin B$, then $H:=B\cap B^g$ is finite, \emph{nontrivial}, and toral in $B$ with $C(H)$ of rank $2$. Further, $C(H) \cap B$ has rank $1$, so $C(H)$ and $B$ are nonconjugate Borel subgroups.
\end{lemma}
\begin{proof}
By Corollary~\ref{cor.SimpleMaxCenterRankFour} and Proposition~\ref{prop.SimpleRankTwoSubRankFour}, the action of $G$ on $X:=B\backslash G$ is $2$-transitive, and by rank considerations, the $2$-point stabilizers in this action are finite. Thus, $H$ is finite. Now, if $H=1$, then $G$ is \emph{sharply} $2$-transitive on $X$. By Proposition~\ref{prop.SimpleRankTwoSubRankFour}, $B$ certainly contains an involution, so we may apply \cite[Proposition~11.71]{BoNe94} to see that $G$ splits, a contradiction.

We now show that $H$ is toral in $B$. Let $h\in H$ be nontrivial. Then $h$ must have an infinite centralizer in each of $B$ and $B^g$. Since $H$ is finite, $C(h)$ must have rank $2$, so $h$ is toral by Corollary~\ref{cor.torality}. In particular, we find that $H$ intersects the unipotent radical of $B$ trivially, so $H$ embeds into a decent torus. Thus, $H$ is cyclic, so $H$ is contained in any torus of $B$ that contains a generator of $H$. 

Finally, we note that $C(H)$ contains two distinct tori coming from $B$ and $B^g$, so $C(H)$ has rank $2$. As we observed early on that the intersection of $B$ with any of its conjugates is finite,  $C(H)$ is not a conjugate of $B$.
\end{proof}

\begin{corollary}
$G$ has odd type.
\end{corollary}
\begin{proof}
Assume that $G$ has even type. By Proposition~\ref{prop.SimpleRankTwoSubRankFour} and Fact~\ref{fact.rankTwoGroups}, the Borel subgroups of $G$ are precisely the normalizers of connected Sylow $2$-subgroups of $G$, so the conjugacy of connected Sylow $2$-subgroups implies that all Borel subgroups are conjugate. By Lemma~\ref{lem.BF}, $G$ has a Borel subgroup with nontrivial center, so the Lemma~\ref{lem.TwoTransTwoPointStab} provides a contradiction.  
\end{proof}

%%% SUBSECTION %%%%%%%%%%%%%%%%%%%%%%%%%%%%%%%%%%%%%%%%%%%%%%
\subsection{Strongly real elements}
%%%%%%%%%%%%%%%%%%%%%%%%%%%%%%%%%%%%%%%%%%%%%%%%%%%%%%%%%
We now freely and frequently use that $G$ has odd type. The end game, which will come soon, exploits the action of $G$ on its set of involutions $I$. Since $G$ must have Pr\"ufer $2$-rank equal to $1$, Fact~\ref{fact.torality} implies that this action is transitive, and our work in the previous subsection, together with Corollary~\ref{cor.SimpleMaxCenterRankFour}, shows that the action is in fact $2$-transitive. However, we will also need some information about strongly real elements of $G$. 

We begin with a lemma. Recall that the $2$-rank of $G$, denoted $\m_2 G$, is the maximal dimension over $\operatorname{GF}(2)$ of an elementary abelian $2$-subgroup.

\begin{lemma}\label{lem.mTwo}
The $2$-rank of $G$ is $1$, so no involution is strongly real.
\end{lemma}
\begin{proof}
If $i$ and $j$ are commuting involutions of $G$, then $j\in C(i)$. By Proposition~\ref{prop.SimpleRankTwoSubRankFour} and Lemma~\ref{lem.BF}, $C(i)$ is nonnilpotent and connected of  rank $2$. By Fact~\ref{fact.rankTwoGroups}, $\m_2 C(i) = 1$, so $i=j$.
\end{proof}

Before we proceed, note that $2$-transitivity of $G$ on $I$ implies that the strongly real elements of $G$ are all conjugate.

\begin{lemma}
Let $r\in G$ be strongly real. Then $C^\circ(r)$ is unipotent of rank $1$ with $r\in C^\circ(r)$, and consequently, $r$ lies in a unique Borel subgroup.
\end{lemma}
\begin{proof}
First, towards a contradiction, assume that $A:=C^\circ(r)$ has rank $0$. In this case, the set $r^G$ is generic in $G$. However, as mentioned in the proof of Proposition~\ref{prop.SimpleRankTwoSubRankFour}, if $T$ is a maximal decent torus of $G$, then $\bigcup T^G$ is  generic. As $G$ is connected,  $r^G$ and $\bigcup T^G$ must intersect nontrivially. Since the elements of $r^G$ have finite centralizers and the elements of $\bigcup T^G$ do not, we have a contradiction, so $\rk A>0$. Now assume that $A$ has rank $2$. Write $r:=ij$ for involutions $i$ and $j$. Certainly $i$ normalizes $A$, so $i \in A$ by the self-normalization of Borel subgroups. This implies that $i$ commutes with $r$, hence with $j$, and this contradicts the fact that $\m_2 G = 1$. Thus, $\rk A = 1$. 

Now, assume that $A$ is a decent torus. As $i$ normalizes $A$, $i$ must centralize the unique involution of $A$. Since $\m_2 G = 1$, $i\in A$, but the same clearly also holds for $j$. As $i\neq j$, $A$ must be unipotent.

To prove that $r\in C^\circ(r)$, it suffices to prove it for some strongly real $s$. Let $H:=C(i) \cap C(j)$. We know that $C(H)$ contains $i$, and by Lemma~\ref{lem.TwoTransTwoPointStab}, it has rank $2$. As $i$ is not central in $C(H)$, $i$ inverts the unipotent radical of $C(H)$. Thus, every $s$ in this unipotent radical is strongly real, and clearly $s\in C^\circ(s)$.

For the final point, observe that $C^\circ(r)$, hence $r$, must be contained in some Borel subgroup by Proposition~\ref{prop.SimpleRankTwoSubRankFour}, and this Borel must be $N(C^\circ(r))$. Now if an arbitrary Borel $B$ contains $r$, it must be that $C_B(r)$ is infinite. Thus, $C^\circ(r) < B$, so $B = N(C^\circ(r))$.
\end{proof}

The previous lemma yields the following essential ingredient for our proof of Theorem~\ref{thm.A}.

\begin{lemma}\label{lem.stronglyRealBorel}
Let $B$ be a Borel subgroup of $G$.  If some nontrivial element of $B$ is inverted by an involution $i$, then $i\in B$. 
\end{lemma}
\begin{proof}
Suppose that some nontrivial element $r$ of $B$ is inverted by an involution $i$. If $r=i$, there is nothing to show, so we may assume that $r$ is strongly real. In this case, $B$ is the unique Borel subgroup containing $r$,  so $i$ normalizes $B$. By the self-normalization of Borel subgroups, $i\in B$.
\end{proof}

%%% SUBSECTION %%%%%%%%%%%%%%%%%%%%%%%%%%%%%%%%%%%%%%%%%%%%%%
\subsection{The proof of Theorem~\ref{thm.A}}
%%%%%%%%%%%%%%%%%%%%%%%%%%%%%%%%%%%%%%%%%%%%%%%%%%%%%%%%%
\begin{proof}[Proof of Theorem~\ref{thm.A}]
We continue to assume that $G$ is a simple group of rank $4$ with a definable subgroup of rank $2$. Thus, all of the results from this section apply to $G$, and we are aiming for a contradiction. Our approach is inspired by \cite[Proposition~11.71]{BoNe94}. We build a point-line geometry. 

Let $\points$ be the set of involutions of $G$. For distinct $i,j\in \points$ define the line through $i$ and $j$ to be $\ell_{ij} := \{ k\in \points : (ij)^k = ji\}$, and set $\lines := \{\ell_{ij} : i,j \in \points \text{ with } i \neq j\}$. A point is incident with a line precisely when it is contained in the line. Note that $\lines$ can be identified with the set of strongly real elements modulo the relation that identifies two strongly real elements if and only if they define the same line. Let us give another characterization of $\ell_{ij}$. Set $r:= ij$, and let $\ell := \ell_{r}$. For $k \in \points$, we claim that $k\in \ell$ if and only if $k \in N(C^\circ(r))$. Clearly every $k$ in $\ell$ is in $N(C^\circ(r))$. Now, as $i\in N(C^\circ(r))$ and $i$ inverts $r$, we see that $N(C^\circ(r))$ does not have a central involution. Thus every involution of $N(C^\circ(r))$ inverts $C^\circ(r)$, which contains $r$, and the claim holds. Additionally, this allows us to see that the (setwise) stabilizer of $\ell$ is $G_\ell = N(C^\circ(r))$.

We now gather some basic information about the geometry. We already know that $G$ acts $2$-transitively on  $\points$, with $\points$ connected of rank $2$, and $G$ acts transitively on $\lines$ as well. Thus, with our above observation that $G_\ell = N(C^\circ(r))$, we find that $\lines$ is also connected of rank $2$, and it is not hard to see that the rank of the point-row  $\points(\ell)$ is $1$. We now claim that two distinct points $i$ and $j$ lie on a unique line, namely $\ell_{ij}$. Suppose that $i,j \in \ell_s$ for $s$ strongly real. Then $i,j \in  N(C^\circ(s))$, so $r:= ij$ is in $N(C^\circ(s))$ as well. Of course, $r$ must have an infinite centralizer in $N(C^\circ(s))$, so $C^\circ(r)  = C^\circ(s)$. Thus, $\ell_{ij} = \ell_s$, so two distinct points lie on a unique line. This can be used to define an equivalence relation on $\points - \{i\}$ by $p\sim q$ if and only if $\ell_{ip} = \ell_{iq}$. As $\rk \points(\ell) = 1$, we find that the line-pencil $\lines(i)$ has rank $1$ as well. We summarize our findings.
\begin{enumerate}
\item Both $\points$ and $\lines$ are connected of rank $2$.
\item Every point-row and every line-pencil has rank $1$. 
\item Two distinct points lie on a unique line.
\end{enumerate}

Now, fix a line $\ell$. Let $\lines_0$ be the set of lines intersecting  $\ell$, and let $\lines_1$ be the set of lines not intersecting $\ell$. We will show that $\lines_0$ and $\lines_1$ both have full rank in $\lines$, and this will be our contradiction. We begin by computing the rank of $\lines_0$. Since distinct points lie on a unique line and every line-pencil has rank $1$, we  conclude that $\rk \lines_0 = \rk \points(\ell) + 1 = 2$. Now let $\points_1$ be the set of points that do not lie on $\ell$. To compute the rank of $\lines_1$, we show that the definable function $\points_1\rightarrow \lines:i\mapsto \ell^i$ has range in $\lines_1$ and is injective.  Let $i$ and $j$ be arbitrary distinct points not lying on $\ell$. We claim that $\ell^i \cap \ell = \emptyset$ and $\ell^i \neq \ell^j$. Write $\ell = \ell_s$ for $s$ strongly real. If $\ell^i$ and $\ell$ intersect, then either they intersect in a unique point centralized by $i$ or the lines are identical. As $\m_2 G = 1$, it must be that $\ell^i=\ell$, so $i \in G_\ell$. As observed above, this implies that $i\in \ell$, so we conclude that $\ell^i \cap \ell = \emptyset$. Next, assume that $\ell^i = \ell^j$. Then $ij \in G_\ell$, so  Lemma~\ref{lem.stronglyRealBorel} implies that $i,j\in G_\ell$. This is again a contradiction, so $\ell^i \neq \ell^j$. In summary, we have a definable injective function from the rank $2$ set $\points_1$ to $\lines_1$, so  $\lines_1$ must also have rank $2$.
\end{proof}

%%% SECTION %%%%%%%%%%%%%%%%%%%%%%%%%%%%%%%%%%%%%%%%%%%%%%%%%
\section{Connected groups of rank $4$}\label{sec.corollaryA}
%%%%%%%%%%%%%%%%%%%%%%%%%%%%%%%%%%%%%%%%%%%%%%%%%%%%%%%%%
%%%%%%%%%%%%%%%%%%%%%%%%%%%%%%%%%%%%%%%%%%%%%%%%%%%%%%%%%
We now address Corollary~\ref{cor.A}. The case when $F(G)$ has rank $2$ is simply Fact~\ref{fact.rankTwoGroups}. Also, we note that the final statement of Corollary~\ref{cor.A} follows from the fact that simple bad groups do not have involutory automorphisms  (see \cite[Theorem~13.3]{BoNe94}) together with the main result of \cite{BBC07}. 

\begin{setupCA}
Let $G$ be a connected group of rank $4$.
\end{setupCA}

\begin{proposition}
If $\rk F(G) = 0$, then either
\begin{enumerate}
\item $G$ is a quasisimple bad group, or
\item $G$ has a normal quasisimple bad subgroup of rank $3$. 
\end{enumerate}
\end{proposition}
\begin{proof}
Since $F^\circ(G)$ is trivial, \cite[Proposition~7.3]{ABC08} ensures that $G$ has a component $Q$, i.e. $Q$ is subnormal and quasisimple. Additionally, $Q$ is definable, and as $G$ is connected, $Q$ is normal. Certainly, $Q$ must have rank at least $3$, and $Z(Q)$ is finite. If $Q$ has rank $4$, then $G=Q$, and Theorem~\ref{thm.A} applies to $G/Z(G)$.

Now assume that  $Q$ has rank $3$. Then either $Q$ has no definable corank $1$ subgroups and $Q$ is a bad group, or $Q/Z(Q)$ is of the form $\psl_2$. However, in the latter case, we find that $Q$ is of the form $\pssl_2$ and $G=QC(Q)$, see \cite[II,~Corollary~2.26 and Proposition~3.1]{ABC08}. Certainly $C(Q)$ has rank $1$, so $C^\circ(Q) \le F(G)$ which is a contradiction. Thus, if $Q$ has rank $3$, then $G$ is an extension of a quasisimple bad group of rank $3$. 
\end{proof}

\begin{proposition}
If $\rk F(G) = 1$, then either
\begin{enumerate}
\item $G$ is a quasisimple bad group, or
\item $G = F(G) * Q$ for some quasisimple subgroup $Q$ of rank $3$.
\end{enumerate}
\end{proposition}
\begin{proof}
Set $F:= F^\circ(G)$. We claim that $G$ is not solvable. If $G$ is solvable, we may use \cite[Proposition~4.11]{ABC08} to linearize the action of $G$ on $F$, but as $F$ has rank $1$, the image of $G$ in $\emorph(F)$ has rank at most $1$. However, $G$ is solvable, so $C^\circ(F) \le F$ by \cite[Proposition~7.3]{ABC08}. Thus $G$ is not solvable, so $G/F(G)$ is either a simple bad group or of the form $\psl_2$.

We now show that $F$ is central. Suppose not. Then there is an $x\in F$ for which $x^G$ has rank $1$. Let $N$ be the kernel of $G$ acting on $x^G$. Of course $N$ has rank at most $3$, and if $N$ has rank $2$ or $3$, we  find that $G$ is solvable. Thus, Fact~\ref{fact.Hru} implies that $G/N \cong \psl_2(K)$ for some algebraically closed field $K$, and $N^\circ = F$. Let $A_1$ and $A_2$ be the connected components of the preimages in $G$ of $2$ distinct unipotent subgroups of $G/N$, and note that $G = \langle A_1,A_2\rangle$. Since each $A_i$ is a rank $2$ group with a rank $1$ unipotent quotient, it must be that both are nilpotent. Thus, $F$ is central in both $A_1$ and $A_2$, so $F$ is central in $G$. 

Now, we again appeal to \cite[Proposition~7.3]{ABC08} to see that $G$ must contain some component $Q$, and one finds that $G= F * Q$. It remains to show that $G$ is bad when $Q=G$. Assume $Q=G$. First note that if $G/Z(G)$ is of the form $\psl_2$, then \cite[II, Proposition~3.1]{ABC08} implies that $G$ has a finite center. Thus, it must be that $G/Z(G)$ is a simple bad group. Let $A$ be a proper connected subgroup of $G$. We aim to show that $A$ is nilpotent and conclude that $G$ is a bad group. Let $B:=AF$. If $B=G$, then $A$ is normal in $G$ contradicting the fact that $G$ is quasisimple. Since $G/Z(G)$ is a bad group of rank $3$, $B/F$ must have rank at most $1$. Thus, $B$ is nilpotent, and the same is true of $A$.   
\end{proof}

%%% SECTION %%%%%%%%%%%%%%%%%%%%%%%%%%%%%%%%%%%%%%%%%%%%%%%%%
\section{Actions on sets of rank $2$}
%%%%%%%%%%%%%%%%%%%%%%%%%%%%%%%%%%%%%%%%%%%%%%%%%%%%%%%%%
%%%%%%%%%%%%%%%%%%%%%%%%%%%%%%%%%%%%%%%%%%%%%%%%%%%%%%%%%
Finally, we address Corollary~\ref{cor.B}.

\begin{proof}[Proof of Corollary~\ref{cor.B}]
Let $G$ satisfy the hypotheses of Corollary~\ref{cor.B}, and let $X$ be a definable set of rank $2$ on which $G$ acts faithfully, definably, transitively, and generically $2$-transitively. Fix $x\in X$. Generic $2$-transitivity implies that $G$ has involutions, so by Corollary~\ref{cor.A}, $Z:=F^\circ(G)$ has rank $1$. Since $Z$ is central and the action is faithful, $G_x \cap Z = 1$, and we see that $B:=ZG_x$ has rank $3$. By Corollary~\ref{cor.A}, $G = Z* Q$ with $Q$ quasisimple of the form $\pssl_2$.

Since $Z$ is normal, the orbits of $Z$ on $X$ determine a quotient $\modu{X}$ of $X$ where the stabilizer of $xZ$ is $B$. Let $N$ be the kernel of the action on $\modu{X}$, and note that $Z$ is contained in the kernel. Since groups of rank $2$ are solvable and $G$ is nonsolvable, the only possibility is that $N$ has rank $1$. By Fact~\ref{fact.Hru},  $G/N$ is of the form $\psl_2$, and the action of $Q/Z(Q)$ on $\modu{X}$ is  equivalent to that of $\psl_2(K)$ on $\proj^1(K)$ for some algebraically closed field $K$. 

Next we show that $Z\cong K^\times$. Set $H:=G^\circ_x$. By generic $2$-transitivity, $G$ is generated by $H$ and any generic conjugate of $H$, so $H$ is not contained in $Q$. Now $H/(H\cap N) \cong HN/N$, and as the latter group is equal to a rank $2$ subgroup of $G/N\cong \psl_2(K)$, the structure of $\psl_2(K)$ implies that $H/(H\cap N)$ is isomorphic to a Borel subgroup of $\psl_2(K)$. Now, $H\cap Q$ is normal in $H$ and of rank $1$, so $H\cap Q$ contains the unipotent radical of $H$. Thus, $H/(H\cap Q) \cong K^\times$ by Lemma~\ref{lem.DivisibleFiniteQuotient}. Now, $H/(H\cap Q) \cong G/Q$, and the latter is isomorphic to $Z/Z\cap Q$. By Corollary~\ref{cor.DivisibleFiniteQuotient}, $Z\cong K^\times$.

It remains to show that $Z(G) = Z$. Let $T$ be a rank $1$ decent torus in $G_x$, and notice that $C(T) = C^\circ(T) =T \times Z$, using Fact~\ref{fact.centTori}. Since $T < G_x$, $T \cap Z(G)  = 1$, and we find that $Z(G)$ must be connected.
\end{proof}

%%% SECTION %%%%%%%%%%%%%%%%%%%%%%%%%%%%%%%%%%%%%%%%%%%%%
\section*{Acknowledgments}
The author is grateful to Adrien Deloro for several valuable conversations about small groups. The author is also grateful to the referee for many  helpful recommendations that greatly improved the efficiency and clarity of the article. Additionally, the author would like to acknowledge the warm hospitality of Universit\"at M\"unster where  the results of the article where obtained. 
%%%%%%%%%%%%%%%%%%%%%%%%%%%%%%%%%%%%%%%%%%%%%%%%%%%%%
%%%%%%%%%%%%%%%%%%%%%%%%%%%%%%%%%%%%%%%%%%%%%%%%%%%%%

%%% REFERENCES %%%%%%%%%%%%%%%%%%%%%%%%%%%%%%%%%%%%%%%%%%%%%
\bibliographystyle{alpha}
\bibliography{WisconsBib}
\end{document}